\documentclass[letterpaper,11pt]{amsart}

\usepackage{amsmath,amssymb,amsthm,pinlabel,tikz}
\usepackage{verbatim,mathtools}
\usepackage{tikz-cd}
\usepackage{float}
\newcommand{\nc}{\newcommand}
\nc{\dmo}{\DeclareMathOperator}
\dmo{\ra}{\rightarrow}
\dmo{\N}{\mathbb{N}}
\dmo{\F}{\mathbb{F}}
\dmo{\Z}{\mathbb{Z}}
\dmo{\R}{\mathbb{R}}
\dmo{\C}{\mathcal{C}}
\dmo{\AC}{\mathcal{AC}}
\dmo{\Mod}{Mod}
\dmo{\PMod}{PMod}
\dmo{\B}{B}
\dmo{\PB}{PB}
\dmo{\GL}{GL}
\dmo{\SL}{SL}
\dmo{\Sp}{Sp}
\dmo{\I}{\mathcal{I}}
\dmo{\el}{\ell_{\C}}
\dmo{\NN}{\mathcal{N}}
\dmo{\Tr}{Tr} 
\dmo{\rk}{rk}
\dmo{\Aut}{Aut}
\dmo{\Inn}{Inn}
\dmo{\Teich}{Teich}
\dmo{\Ind}{Ind}
\dmo{\cd}{cd}

\dmo{\forget}{Forget}
\dmo{\Homeo}{Homeo}
\dmo{\Out}{Out}
\dmo{\Diffeo}{Diffeo}
\dmo{\push}{Push}
\dmo{\capd}{Cap}
\dmo{\CG}{CG}
\dmo{\UCG}{UCG}
\dmo{\FCG}{FCG}
\dmo{\PFB}{PFB}
\dmo{\BG}{B}
\dmo{\PBG}{PB}
\dmo{\NRH}{NRH}

\dmo{\st}{st}
\dmo{\lk}{lk}

\nc{\nt}{\newtheorem}

\newtheorem{thm}{{\bf Theorem}}[section]
\newtheorem{lem}[thm]{{\bf Lemma}}

\newtheorem{prop}[thm]{{\bf Proposition}}

\numberwithin{equation}{section}

\newtheorem*{rep@theorem}{\rep@title}
\newcommand{\newreptheorem}[2]{%
\newenvironment{rep#1}[1]{%
 \def\rep@title{#2 \ref{##1}}%
 \begin{rep@theorem}}%
 {\end{rep@theorem}}}

\newreptheorem{theorem}{Theorem}
\makeatletter
\newcommand*{\rom}[1]{\expandafter\@slowromancap\romannumeral #1@}
\makeatother

\theoremstyle{definition}
\newtheorem*{ack}{Acknowledgements}

\title[(Outer) automorphism groups of RAAGs are almost NRH]
{Automorphism and outer automorphism groups of Right-angled Artin groups are not relatively hyperbolic}

\author{Junseok Kim}

\address{%
        Department of Mathematical Sciences, KAIST,
		291 Daehak-ro Yuseong-gu, Daejeon, 34141, South Korea 
}
\email{%
        jsk8818@kaist.ac.kr
}

\author{Sangrok Oh}

\address{%
        Department of Mathematics, University of the Basque Country,
        Barrio Sarriena, Leioa, 48940, Spain
}
\email {%
sangrokoh.math@gmail.com}

\author{Philippe Tranchida}
\address{%
		Department of Mathematical Sciences, KAIST,
		291 Daehak-ro Yuseong-gu, Daejeon, 34141, South Korea 
}
\email{%
        ptranchi@kaist.ac.kr
}

\begin{document}

\maketitle

\begin{abstract}
We show that the automorphism groups of right-angled Artin groups whose defining graphs have at least 3 vertices are not relatively hyperbolic. We then show that the outer automorphism groups are not relatively hyperbolic, if they are not virtually isomorphic to a right-angled Artin group whose defining graph is either a single vertex or disconnected. 
\end{abstract}

\section{Introduction}
Associated with a finite simplicial graph $\Gamma$ whose vertex set and edge set are $V$ and $E$, respectively, is the right-angled Artin group (RAAG) $A_\Gamma$ which is defined by the following group presentation: $$A_\Gamma = \langle v \in V \mid [u, v] = 1 ~\text{for}~\{u, v\} \in E \rangle.$$
In these settings, $\Gamma$ is said to be the defining graph of $A_\Gamma$.
As extreme examples, RAAGs can be free abelian groups $\Z^n$, when the defining graphs are complete, or free groups $\F_n$, when the defining graphs have no edges.
In contrast, generic RAAGs have interesting behaviors; for example, some of their subgroups may not be isomorphic to RAAGs. Subgroups of RAAGs, such as Bestvina--Brady groups \cite{Bestvina1997MorseTA}, are actually quite wild and have been used to construct examples of groups with peculiar properties. For a brief introduction to RAAGs, we refer to Charney's note \cite{Cha}.

\thispagestyle{empty}
In this paper, we look at the automorphism and outer automorphism groups of $A_\Gamma$ denoted by $\Aut(A_\Gamma)$ and $\Out(A_\Gamma)$, respectively (the inner automorphism group is denoted by $\Inn(A_\Gamma)$). $\Out(\Z^n)$ will usually be identified with  $\GL_n(\Z)$.
Even though $\Z^n$ and $\F_n$ have a lot of opposite properties in the algebraic sense, $\GL_n(\Z)$ and $\Out(\F_n)$ share many common properties: for example, both of them are virtually torsion-free, residually finite, and have finite virtual cohomological dimension. Charney and Vogtmann extended these results to every $\Out(A_{\Gamma})$ in their papers \cite{charney2009finiteness} and \cite{charney2011subgroups}.

Another interesting common feature of $\GL_n(\Z)$ and $\Out(\F_n)$ is that they are not relatively hyperbolic, except when $n$ is small enough, in which case they are actually hyperbolic.
Anderson--Aramayona--Shackleton \cite{simplecriterion} established a simple criterion for showing non-relative hyperbolicity of groups generated by infinite order elements.
Using this criterion, they proved that, as long as $n \geq 3$, $\GL_n(\Z)$ and $\Out(\F_n)$, even $\Aut(\F_n)$ are not relatively hyperbolic; see also Behrstock--Druţu--Mosher \cite{BDM09}.
It is then quite natural to ask whether $\Aut(A_\Gamma)$ and $\Out(A_\Gamma)$ are always not relatively hyperbolic. This turns out to be true, except for a few cases.

Here are the two main theorems of this paper. 

\begin{reptheorem}{Main1}
If a finite simplicial graph $\Gamma$ contains at least three vertices, then the automorphism group of the right-angled Artin group of $\Gamma$ is not relatively hyperbolic.
\end{reptheorem}

\begin{reptheorem}{Main2}
If the outer automorphism group of a right-angled Artin group is infinite and relatively hyperbolic, then it is virtually a right-angled Artin group whose defining graph is either a single vertex or disconnected.
\end{reptheorem}

We remark that even though $\Aut(A_{\Gamma})$ is almost never relatively hyperbolic, Genevois proved in \cite{genevois2020automorphisms} that $\Aut(A_{\Gamma})$ is acylindrically hyperbolic if and only if $\Gamma$ is not a join and contains at least $2$ vertices.

The definition and study of relatively hyperbolic groups come from the following observation: even when a group $G$ fails to be hyperbolic, it might still exhibit hyperbolic behaviours if we look only ``outside" of some proper subgroups, called parabolic subgroups. 
With this observation in mind, one obstruction for being relatively hyperbolic is the existence of a specific collection $\mathcal{A}$ of proper subgroups which are far from being hyperbolic (for example, free abelian subgroups) and are well-networked. 
The term `well-networked' means that (1) the union of all the subgroups in $\mathcal{A}$ generates a finite index subgroup of $G$, (2) for any $A,A'\in\mathcal{A}$, there exists a sequence $A_1=A,\cdots,A_n=A'$ such that $A_i\cap A_{i+1}$ is infinite. 
If we find such a collection of subgroups of $G$, then $G$ would be \emph{never} relatively hyperbolic regardless of the choice of parabolic subgroups.

Following the above idea, the notion of the \emph{commutativity graph} of a group $G$ is recalled in Section \ref{2.2} as a tool to show that $G$ is not relatively hyperbolic.
One of the main assumptions to define the commutativity graph is the existence of a (possibly infinite) generating set of $G$ which consists of infinite order elements. However, there are finite order elements in the usual generating sets of $\Aut(A_\Gamma)$ and $\Out(A_\Gamma)$.
To handle this problem, in Section \ref{2.1}, we find a finite index subgroup which is generated by a finite collection of infinite order elements. 
In Section \ref{3} and \ref{4}, we prove that $\Aut(A_\Gamma)$ and $\Out(A_\Gamma)$ are in general not relatively hyperbolic, respectively, by using those finite-index subgroups with the fact that being (or not being) relatively hyperbolic is a quasi-isometry invariant.

\begin{ack} We thank our advisor Hyungryul Baik for useful comments, Donggyun Seo for useful conversations, and Anthony Genevois for giving us helpful feedback. The first and third authors were partially supported by Samsung Science \& Technology Foundation grant No. SSTF-BA1702-01 and the National Research Foundation of Korea(NRF) grant funded by the Korea government(MSIT) (No. 2020R1C1C1A01006912).
\end{ack}

\section{Preliminaries}\label{2}
In this paper, $\Gamma$ is always assumed to be a finite simplicial graph with vertex set $V$.  
For a vertex $v \in V$, the \emph{link} of $v$, denoted by $\lk(v)$, is the full subgraph of $\Gamma$ spanned by vertices adjacent to $v$. 
Similarly, the \emph{star} of $v$, denoted by $\st(v)$, is the full subgraph of $\Gamma$ spanned by vertices adjacent to $v$ and $v$ itself. 
We then say that $v \leq w$ if $\lk(v) \subset \st(w)$. This partial order induces an equivalence relation on $V$ by setting $v \sim w$ if $v\leq w$ and $w \leq v$. 
The partial order then descends to a partial order on the collection of the equivalence classes of vertices by setting $[v] \leq [w]$ if for some, and thus all, representatives $v' \in [v]$ and $w' \in [w]$, we have $v' \leq w'$. A vertex $v \in V$ is maximal if any vertex $w$ such that $v \leq w$ is actually equivalent to $v$.

\subsection{(Outer) automorphism groups of RAAGs}\label{2.1}
A theorem which was conjectured by Servatius \cite{SERVATIUS} and proved by Laurence \cite{Laurence} says that $\Aut(A_{\Gamma})$ is generated by the following four classes of automorphisms:
\begin{itemize}
    \item \textbf{Graph automorphisms}. An automorphism of $\Gamma$ induces an automorphism of $A_{\Gamma}$ since it preserves the edges of $\Gamma$ and thus the relations of $A_{\Gamma}$. The automorphism obtained this way is called a \emph{graph automorphism}.
    \item \textbf{Inversions}. An automorphism of $A_{\Gamma}$ by sending one generator $v$ to its inverse $v^{-1}$ is called an \emph{inversion}.
    \item \textbf{Transvections}. Take two vertices $v$ and $w$ in $\Gamma$ such that $v \leq w$. Then the automorphism sending $v$ to $vw$ and fixing all the other vertices is called a \emph{right transvection} and is denoted by $R_
    {vw}$. We can similarly define a \emph{left transvection} $L_{vw}$ by sending $v$ to $wv$ and still fixing all the other vertices.
    \item \textbf{Partial conjugations}. For a vertex $v\in\Gamma$, let $C$ be a connected component of $\Gamma - \st(v)$. The automorphism defined by conjugating every vertex in $C$ by $v$ is called a \emph{partial conjugation} and is denoted by $P_v^C$. If a component $C$ of $\Gamma - \st(v)$ is composed of a single vertex $w$, we write $P_v^w$ instead of $P_v^{\{w\}}$.
\end{itemize}

Note that inversions and graph automorphisms have finite order, but transvections and partial conjugations have infinite order.
Let $\Aut^*(A_{\Gamma})$ be the subgroup of $\Aut(A_{\Gamma})$ generated only by transvections and partial conjugations. Then we can easily deduce the following:

\begin{lem}\label{FI}
$\Aut^*(A_{\Gamma})$ is a finite index normal subgroup of $\Aut(A_{\Gamma})$.
\end{lem}
\begin{proof}
The fact that $\mathrm{Aut}^*(A_{\Gamma})$ is a normal subgroup can be shown by checking that the conjugate of a transvection or a partial conjugation by an inversion or a graph automorphism is still a transvection or a partial conjugation. 
Since the image of the quotient map $\mathrm{Aut}(A_{\Gamma})\rightarrow \mathrm{Aut}(A_\Gamma)/\mathrm{Aut}^*(A_{\Gamma})$ is generated by the images of graph automorphisms and inversions, its cardinality is smaller than or equal to the one of the subgroup generated by inversions and graph automorphisms, which is finite. We can thus deduce that $\mathrm{Aut}^*(A_{\Gamma})$ is of finite index.
\end{proof}

Under the quotient map $\Aut(A_\Gamma)\rightarrow \Out(A_\Gamma):=\Aut(A_\Gamma)/\Inn(A_\Gamma)$, $\Out(A_\Gamma)$ is generated by the images of graph automorphisms, inversions, transvections and partial conjugations.
(Note that some images of partial conjugations may be trivial in $\Out(A_\Gamma)$.)
Similarly, $\Out^*(A_{\Gamma})$ is defined to be the subgroup generated by the images of transvections and partial conjugations in $\Out(A_\Gamma)$; it can be considered as the image of $\Aut^*(A_{\Gamma})$ in $\Out(A_{\Gamma})$. By the above lemma, $\Out^*(A_{\Gamma})$ is also a finite index normal subgroup of $\Out(A_{\Gamma})$.

\subsection{Non-relative hyperbolicity}\label{2.2}
We first recall the definition of relative hyperbolicity due to Bowditch \cite{Bow}.
Let $G$ be a finitely generated group and $\mathcal{H}$ a finite collection of proper finitely generated subgroups of $G$.
Choose a finite generating set $S$ of $G$ and consider the Cayley graph $\Lambda = \Lambda(G,S)$.
The \emph{coned-off Cayley graph $\hat{\Lambda}(G,\mathcal{H})$} is defined as follow: starting with the Cayley graph $\Lambda$, for each coset $gH_i$ with $g\in G$, $H_i\in\mathcal{H}$, we add a vertex $v(gH_i)$ to $\Lambda$ and connect $v(gH_i)$ by an edge to each vertex in $gH_i$.
We then say that $G$ is \emph{relatively hyperbolic with respect to $\mathcal{H}$} if 
\begin{itemize}
    \item the coned-off Cayley graph $\hat{\Lambda}(G,\mathcal{H})$ is $\delta$-hyperbolic and
    \item $\hat{\Lambda}(G,\mathcal{H})$ is \emph{fine}. This means that for each integer $k$, all edges $e$ of $\hat{\Lambda}(G,\mathcal{H})$ are contained in finitely many simple cycles of length $k$.
\end{itemize}
Whenever $\hat{\Lambda}(G,\mathcal{H})$ satisfies the first of the above conditions, an element in $\mathcal{H}$ is said to be a \emph{parabolic\ subgroup} of $G$ and $G$ is said to be \emph{weakly relatively hyperbolic} (w.r.t. $\mathcal{H}$).
If $G$ is not relatively hyperbolic with respect to any choice of a finite collection of proper finitely generated subgroups $\mathcal{H}$, then it is said to be \emph{not relatively hyperbolic}.

There are two necessary conditions for parabolic subgroups of relatively hyperbolic groups. Let $G$ be a finitely generated group which is relatively hyperbolic with respect to a finite collection $\mathcal{H}=\{H_i\}$ of parabolic subgroups.
The first one is Theorem 1.4 in \cite{Osin06} about virtual malnormality of parabolic subgroups. 
\begin{thm}[\cite{Osin06}]\label{VirtualMalnormal}
For $H_i,H_j\in\mathcal{H}$ and $g_1,g_2\in G$, $g_1H_ig^{-1}_1\cap g_2H_jg^{-1}_2$ is finite if either $H_i$ and $H_j$ are distinct or $H_i=H_j$ and $g^{-1}_1g_2\notin H_i$.
\end{thm}

In particular, this implies that parabolic subgroups are almost malnormal.

The second one is a statement which slightly generalizes Lemma 5 in \cite{simplecriterion}.

\begin{lem}\label{ParabolicSubgroup}
Suppose $H$ is a subgroup of $G$ isomorphic to a RAAG $A_\Gamma$ whose defining graph $\Gamma$ is connected. Then $H$ is contained in a conjugate of a parabolic subgroup $H_i\in\mathcal{H}$.
\end{lem}
\begin{proof}
This is a direct consequence of Theorem 4.16 and Theorem 4.19 in \cite{Osin06}, stating that a free abelian subgroup of rank 2 has to be contained in a conjugate of $H_i\in\mathcal{H}$. Since the subgroup generated by the end points of each edge of $\Gamma$ is a free abelian subgroup of rank 2, we can deduce that $H$ is contained in a conjugate of a parabolic subgroup. 
\end{proof}

From these necessary conditions of parabolic subgroups, a simple criterion for detecting non-relative hyperbolicity is developed as in \cite{simplecriterion}.
Let $G$ be a group and $S$ be a (possibly infinite) generating set consisting only of infinite order elements. The \emph{commutativity graph} $K(G,S)$ of $G$ with respect to $S$ is the simplicial graph with vertex set $S$ and in which two distinct vertices $s$ and $s'$ are connected by an edge if there exist integers $n_s$ and $n_{s'}$ such that $\langle s^{n_s},(s')^{n_{s'}}\rangle$ is abelian.
The main theorem of \cite{simplecriterion} is then the following: 

\begin{thm}\cite{simplecriterion}\label{Criterion}
Let $G$ be a finitely generated group and $S$ be a (possibly infinite) generating set of $G$ which consists of infinite order elements and contains at least $2$ elements. Suppose that $K(G,S)$ is connected and that there are at least $2$ vertices $s$ and $s'$ in $K(G,S)$ such that $\langle s^{n_s},(s')^{n_{s'}}\rangle$ is a rank two free abelian group for some integers $n_s$ and $n_{s'}$. Then, $G$ is not relatively hyperbolic.
\end{thm}

Finally, due to the result of Druţu, in order to know whether a finitely generated group $G$ is relatively hyperbolic or not, we may look at other groups quasi-isometric to $G$ (for example, finite index subgroups).

\begin{thm}\cite{Drutu06relativelyhyperbolic}\label{QII}
In the class of finitely generated groups, being relatively hyperbolic is a quasi-isometry invariant.
\end{thm}

\section{Automorphism group}\label{3}

The goal of this section is to prove that $\Aut(A_{\Gamma})$ is in general not relatively hyperbolic. 
In order to use Theorem \ref{Criterion}, we work with $\Aut^*(A_{\Gamma})$ instead of $\Aut(A_{\Gamma})$. Indeed, $\Aut^*(A_\Gamma)$ is generated by infinite order elements and is a finite index subgroup, by Lemma \ref{FI}. By Theorem \ref{QII}, it is then enough to show that $\Aut^*(A_{\Gamma})$ is not relatively hyperbolic. 

\begin{thm}\label{Main1}
Let $\Gamma$ be a graph which has at least $3$ vertices and $S$ the set of all transvections and partial conjugations in $\Aut(A_\Gamma)$. Then the commutativiy graph $K(\Aut^*(A_{\Gamma}), S)$ is connected. Hence, $\Aut(A_{\Gamma})$ is not relatively hyperbolic.
\end{thm}

\begin{proof}
The proof is divided into three steps. First, we show that, as long as they exist, any two transvections are joined by a path in $K=K(\Aut^*(A_{\Gamma}),S)$.
Then, we show that the same holds for any two partial conjugations. Finally, we show that any partial conjugation and transvection are joined by a path, as long as they exist.

\vspace{1.5mm}
\noindent\textbf{Claim 1}. If there are at least two distinct transvections, then any two transvections are joined by a path in $K$.
\vspace{1mm}

Let $a$ and $b$ be vertices in $\Gamma$ such that $a \leq b$. If $a$ and $b$ are adjacent, then $R_{ab}=L_{ab}$. Otherwise, $R_{ab}$ and $L_{ab}$ are distinct but $[R_{ab},L_{ab}]=1$, i.e. $R_{ab}$ and $L_{ab}$ are joined by an edge in $K$. (In both cases, $[R_{ab},L_{ab}]=1$.)
Thus, to prove the claim, we only need to show that there is a path in $K$ from $R_{ab}$ to either $R_{cd}$ or $L_{cd}$ for any two vertices $c,d\in\Gamma$ with $c\leq d$. There are five cases we need to handle.
\vspace{0.7mm}

\textbf{1)} If $c = a$, then $R_{ab}$ and $L_{ad}$ are joined by an edge since $[R_{ab},L_{ad}]=1$.
\vspace{0.7mm}

\textbf{2)} If $c \neq a,b$ and $d \neq a$, then $R_{ab}$ and $R_{cd}$ are joined by an edge since $[R_{ab},R_{cd}]=1$.
\vspace{0.7mm}

\textbf{3)} If $c \neq a,b$ and $d = a$, then $R_{ab}$ and $L_{ca}$ are joined by a path. Indeed, $c\leq d = a\leq b$ and thus, $c\leq b$. This means that there is a transvection $R_{cb}$. Then there is a path joining $R_{ab}$ and $L_{ca}$ since $[R_{ab},R_{cb}]=[R_{cb},L_{ca}]=1$.
\vspace{0.7mm}

\textbf{4)} If $c = b$ and $d \neq a,b$, then $R_{ab}$ and $R_{bd}$ are joined by a path. Indeed, $a\leq b\leq d$ and thus, $a\leq d$ so that there is a transvection $L_{ad}$. Then there is a path joining $R_{ab}$ and $R_{bd}$ since $[R_{ab},L_{ad}]=[L_{ad},R_{bd}]=1$.
\vspace{0.7mm}

\textbf{5)} If $c =b$ and $d =a$, then $a\sim b$. There are two cases:\\
\textbf{5-1)} The case that $a$ and $b$ are adjacent, i.e. $\st (a)=\st(b)$: In this case, we show that $R_{ab}$ and $R_{ba}$ are joined by a path. If the star $\st(a)$ does not cover the whole graph $\Gamma$, choose $v\in \Gamma - \st(a)$ and let $\Gamma_0$ be the component of $\Gamma - st(a)$ containing $v$. Note that $\Gamma_0$ does not contain $a$ and $b$. Then we have
\[[R_{ab},P_b^{\Gamma_0}]=[P_b^{\Gamma_0},P_a^{\Gamma_0}]=[P_a^{\Gamma_0},R_{ba}]=1.\]
Otherwise, $\st(a)$ covers the whole graph so that $w\leq a\sim b$ for any vertex $w\neq a,b$ in $\Gamma$. Then we have
\[[R_{ab},R_{wb}]=[R_{wb},L_{wa}]=[L_{wa},R_{wa}]=[R_{wa},R_{ba}]=1.\]
\textbf{5-2)} The case that $a$ and $b$ are not adjacent, i.e. $\lk(a)=\lk(b)$: In this case, we show that $R_{ab}$ and $L_{ba}$ are joined by a path. If the link $\lk(a)$ is empty, then $a\sim b\leq w$ for any vertex $w\neq a,b$ in $\Gamma$. Then we have
\[[R_{ab},L_{aw}]=[L_{aw},R_{bw}]=[R_{bw},L_{ba}]=1.\]
Otherwise, choose $w\in \lk(a)$.
If $\st(w)$ does not cover the whole graph, then \[[R_{ab},P_b^{a}]=[P_b^{a},P_w^{\Gamma'}]=[P_w^{\Gamma'},P_a^{b}]=[P_a^{b},L_{ba}]=1,\]
where $\Gamma'$ is a component of $\Gamma- \st(w)$.
If $\st(w)=\Gamma$, then $a\sim b\leq w$, and so
\[[R_{ab},L_{aw}]=[L_{aw},R_{bw}]=[R_{bw},L_{ba}]=1.\]

By the above five cases, we show that any two transvections (if they exist) are joined by a path in $K$.
\vspace{1mm}

If $\Gamma$ is a complete graph, then there is no partial conjugation so that the theorem holds by the above claim. From now on, therefore, we assume that $\Gamma$ is not complete. In particular, there are at least two partial conjugations.

\vspace{1.5mm}
\noindent\textbf{Claim 2}. Any two partial conjugations $P_a^C$ and $P_b^D$ are joined by a path in $K$ for any choices of $a,b,C,D$.
\vspace{1mm}

Note that $[P_a^{C_1},P_a^{C_2}]=1$ whenever the partial conjugations are defined. It means that to see whether two partial conjugations $P_a^C$ and $P_b^D$ are joined by a path for any two distinct vertices $a$ and $b$, it is enough to check only one particular choice of $C$ and $D$. 
\vspace{1mm}

Suppose $\Gamma$ is connected. If there is no vertex in $\Gamma$ whose star is the whole graph, for any two vertices $a,b\in\Gamma$, $\Gamma - \st(a)=\Gamma_1\sqcup \cdots \sqcup \Gamma_m$ and $\Gamma - \st(b)=\Gamma_1'\sqcup \cdots \sqcup \Gamma_n'$ where each $\Gamma_i$ and $\Gamma_j'$ are components of $\Gamma- \st(a)$ and $\Gamma- \st(b)$, respectively. 
If $[a,b]=1$, this implies that $[P_a^{\Gamma_i},P_b^{\Gamma'_j}]=1$ for any $i$ and $j$ so that the claim holds by the connectivity of $\Gamma$. 

If $\Gamma=\st(b)$ for some vertex $b\in\Gamma$, then one can easily see that $P_a^C$ and $R_{ab}$ commute for any $a\in\lk(b)$ the complement of whose star has a non-empty component $C$. 
By \textbf{Claim 1}, we can deduce that any two partial conjugations are joined by a path.
\vspace{1mm}

Now, suppose $\Gamma$ has at least two components $\Gamma_1$ and $\Gamma_2$, and there are two partical conjugations $P_a^C$ and $P_b^D$ for $a\in\Gamma_1$ and $b\in\Gamma_2$. We only need to show that there are some components $C$ and $D$ of $\Gamma-\st(a)$ and $\Gamma-\st(b)$, respectively, such that $P_a^C$ and $P_b^D$ are joined by a path in $K$. There are two cases depending on the number of vertices in $\Gamma_2$.
\vspace{0.7mm}

\textbf{1)} Suppose $\Gamma_2$ has at least two vertices. There are three cases:\\
(where $\st(b)\subsetneq \Gamma_2$.) Then $[P_a^{\Gamma_2},P_b^D]=1$ where $D$ is a component of $\Gamma_2- \st(b)$. This is because
$$P_a^{\Gamma_2}(P_b^D(s))=P_a^{\Gamma_2}(bsb^{-1})=aba^{-1}\cdot asa^{-1}\cdot ab^{-1}a^{-1}=absb^{-1}a^{-1},$$
and
$$P_b^D(P_a^{\Gamma_2}(s))=P_b^D(asa^{-1})=absb^{-1}a^{-1}$$
for any vertex $s\in D$.\\
(where $\st(b)=\Gamma_2$ but $\Gamma_2$ is not a complete graph.) There are a vertex $b_1\in\Gamma_2$ and a component $D_1$ of $\Gamma_2-\st(b_1)$. 
Then $$[P_b^{\Gamma_1},P_{b_1}^D]=[P_{b_1}^D,P_a^{\Gamma_2}]=1$$ since $b$ and $b_1$ are adjacent.\\
(where $\Gamma_2$ is a complete graph.) Then we have
\[[P_{b}^{\Gamma_1},R_{b_1b}]=[R_{b_1b},P_{a}^{\Gamma_2}]=1\]
for any vertex $b_1\neq b$ in $\Gamma_2$ since $b_1\leq b$. 
\vspace{0.7mm}

\textbf{2)} Suppose $\Gamma_2$ has only one vertex $b$. There are two cases:\\
(where $\Gamma_1$ is not a complete graph.) Let $a_1$ be a vertex in $\Gamma_1$ such that $\st(a_1)\subsetneq \Gamma_1$. If $a=a_1$, then the claim holds since $[P_{a}^{C},P_b^{\Gamma_1}]=1$ where $C$ is a component of $\Gamma_1 - \st(a)$. 
If $a\neq a_1$ but $\st(a)=\Gamma_1$, then we additionally have $[P_{a_1}^{C},P_a^{\Gamma_2}]=1$ so that the claim holds. \\
(where $\Gamma_1$ is a complete graph.) We must divide into two cases again. If $\Gamma_1$ has at least two vertices, then for any vertex $a_1\neq a$ in $\Gamma_1$, $a_1\leq a$ so that
\[R_{a_1a}(P_b^{\Gamma_1}(a_1))=R_{a_1a}(ba_1b^{-1})=ba_1ab^{-1},\]
and
\[P_b^{\Gamma_1}(R_{a_1a}(a_1))=P_b^{\Gamma_1}(a_1a)=ba_1b^{-1}\cdot bab^{-1}=ba_1ab^{-1}.\]
Thus, we have
\[[P_{a}^{\Gamma_2},R_{a_1a}]=[R_{a_1a},P_b^{\Gamma_1}]=1.\]
If $\Gamma_1$ has only one vertex $a$, then we have
\[[P_a^{\Gamma_2},R_{ba}]=[R_{ab},P_{b}^{\Gamma_1}]=1.\]
Since $\Gamma$ has at least 3 vertices, there is a vertex $c$ such that $a\sim b\leq c$. Since there is a path between $R_{ab}$ and $R_{ba}$ by \textbf{Claim 1}, $P_a^{\Gamma_2}$ and $P_b^{\Gamma_1}$ are joined by a path in $K$.

\vspace{1.5mm}
\noindent\textbf{Claim 3}. Any transvection is adjacent to a partial conjugation in $K$.
\vspace{1mm}

Suppose $a$ and $b$ are vertices in $\Gamma$ such that $a\leq b$.
\vspace{0.7mm}

\textbf{1)} If $\st(b)\subsetneq \Gamma$, then $[R_{ab},P_{b}^C]=1$ for any component $C$ of $\Gamma - \st(b)$.
\vspace{0.7mm}

\textbf{2)} Suppose $\st(b)=\Gamma$. Since $\Gamma$ is not complete, there is a vertex $c$ of $\Gamma$ such that $st(c)\subsetneq \Gamma$ ($c$ may be equal to $a$). For any component $C$ of $\Gamma- \st(c)$, we have $[R_{ab},P_c^C]=1$. Therefore, there is an edge joining a transvection and a partial conjugation if they exist.

\vspace{1.5mm}
In summary, if $\Gamma$ is complete, by \textbf{Claim 1}, $K$ is connected. Otherwise, $\Aut^*(A_\Gamma)$ contains at least two partial conjugations. If it has no transvections, by \textbf{Claim 2}, $K$ is connected. If it has a transvection, by combining the three claims, we can show that $K$ is connected. 
\end{proof}

The only cases not covered by the above theorem are those of RAAG's whose defining graphs have 1 or 2 vertices. If $A_{\Gamma}$ is $\mathbb{Z}$, then $\Aut(A_\Gamma) = \Z_2$ is finite. In the remaining case, $A_{\Gamma}$ is either $\Z^2$ or $\F_2$ so that $\Aut(A_\Gamma)$ is either $\GL_2(\Z)$ or $\Aut(\F_2)$, respectively. $\GL_2(\Z)$ is hyperbolic since it is virtually free. For $\Aut(\F_2)$, consider the subgroup $\Aut^+(\F_2)$ which is the preimage of the special linear subgroup $\SL_2(\Z)\subset\GL_2(\Z)$ under the homomorphism $\Aut(\F_2)\rightarrow\GL_2(\Z)$ induced from the abelianization map $\F_2\rightarrow \Z^2$. Then $\Aut^+(\F_2)$ is a finite index subgroup of $\Aut(\F_2)$ and can be shown to be isomorphic to the pure mapping class group of a twice punctured torus, which is not relatively hyperbolic by (the proof of) Theorem 8.1 in \cite{BDM09}.
\vspace{1mm}

We conclude this section with a remark that was pointed out to us by Anthony Genevois. If $\Gamma$ is connected, there is a shorter argument to prove that $\Aut(A_\Gamma)$ is not relatively hyperbolic. 
Suppose that $\Aut(A_\Gamma)$ is relatively hyperbolic with respect to a finite collection $\mathcal{H}=\{H_i\}$ of parabolic subgroups.
If there is a vertex $v\in\Gamma$ which is adjacent to all other vertices, then the subgroup of $\Aut(A_\Gamma)$ generated by all transvections induced by this central vertex $v$ is an infinite normal free abelian subgroup, contained in some $H_i \in \mathcal{H}$ by Lemma \ref{ParabolicSubgroup}.
Otherwise, $\Inn(A_\Gamma)$ is isomorphic to $A_\Gamma$ and it is thus an infinite normal subgroup, contained in some $H_i \in \mathcal{H}$ by Lemma \ref{ParabolicSubgroup}. In both case, we found an infinite normal subgroup contained in an almost malnormal subgroup $H_i$. By Theorem \ref{VirtualMalnormal}, this implies that $H_i = \Aut(A_\Gamma)$, which is a contradiction.

\section{Outer automorphism group}\label{4}
In this section, we look at relative hyperbolicity of $\Out(A_\Gamma)$. In the same spirit as the proof of relative hyperbolicity of $\Aut(A_\Gamma)$, we work with the finite-index subgroup $\Out^*(A_{\Gamma})$ instead of the whole group $\Out(A_{\Gamma})$. Let $S$ be the set of all transvections and partial conjugations in $\Aut(A_\Gamma)$ as before and $S'$ the set of all the (non-trivial) images of elements of $S$ in $\Out(A_{\Gamma})$. 
We want to investigate the connectivity of $K(\Out^*(A_{\Gamma}), S')$. Unfortunately, the proof of the connectivity of $K(\Aut^*(A_{\Gamma}), S)$ does not directly imply that $K(\Out^*(A_{\Gamma}), S')$ is connected since some partial conjugations (indeed, inner automorphisms) in $\Aut(A_{\Gamma})$ are sent to the identity element in $\Out(A_{\Gamma})$. For the rest of this section, when we say a transvection or a partial conjugation, we mean its image in $\Out(A_\Gamma)$.

Before proceeding, we recall some facts about the subgroup $\mathrm{PSA}(A_\Gamma)\leq\Aut(A_\Gamma)$ ($\mathrm{PSO}(A_\Gamma)\leq\Out(A_\Gamma)$, resp.) generated by partial conjugations, which is said to be the \emph{pure symmetric automorphism group} (\emph{pure symmetric outer automorphism group}, resp.) of $A_\Gamma$.
Koban--Piggott showed in \cite{KP} that $\mathrm{PSA}(A_\Gamma)$ has a group presentation whose generators are partial conjugations and relators are commutators.
Moreover, they showed that $\mathrm{PSA}(A_\Gamma)$ is isomorphic to a RAAG if and only if $\Gamma$ has no SIL-pairs (which is defined below Lemma~\ref{lem:SharedComponents}).
With a similar flavor, Day--Wade found the criterion for $\mathrm{PSO}(A_\Gamma)$ to be a RAAG \cite{day2018subspace}. 

In the study of $\mathrm{PSA}(A_\Gamma)$ or $\mathrm{PSO}(A_\Gamma)$, the most important thing is to know when two partial conjugations commute, and there is a precise description using the following fact.

\begin{lem}[Lemma 2.1 in \cite{day2018subspace}]\label{lem:SharedComponents}
Let $a$ and $b$ be nonadjacent vertices of $\Gamma$. Then the components of $\Gamma- \st(a)$ consist of $A_0,\cdots,A_k,C_1,\cdots,C_l$ and the components of $\Gamma-\st(b)$ consist of $B_0,\cdots,B_m,C_1,\cdots,C_l$ where $b\in A_0$ and $a\in B_0$, and $A_1,\cdots,A_k \subset B_0$ and $B_1,\cdots,B_m \subset A_0$.
\end{lem}

In the above lemma, $A_0$ and $B_0$ are called the \emph{dominating components}, $C_i$'s are called the \emph{shared components}, and the other components are called the \emph{subordinate components}.
We say $(a,b)$ is an \emph{SIL-pair} if $l\geq 1$.
Note that any of $k$, $m$ or $l$ can be zero; for instance, $l=0$ implies that there is no shared component.

\begin{lem}[Lemma 2.4 in \cite{day2018subspace}]\label{lem:commutingrelation}
Let $a$ and $b$ be nonadjacent vertices in $\Gamma$ such that there are non-trivial partial conjugations $P_a^C$ and $P_b^D$ in $\Out(A_\Gamma)$. Then $[P_a^C,P_b^D]\neq 1$ in $\Out(A_\Gamma)$ if and only if $(a,b)$ is an SIL-pair and one of the following conditions hold:
\begin{itemize}
\item $C$ and $D$ are the dominating components for the pair $(a, b)$.
\item One of $C$ or $D$ is dominating and the other is shared.
\item $C$ and $D$ are identical shared components.
\end{itemize}
\end{lem}

Now, let us see non-relative hyperbolicity of $\Out^*(A_\Gamma)$ when $S'$ consists of only partial conjugations, i.e. $\Out^*(A_\Gamma)=\mathrm{PSO}(A_\Gamma)$, by examining $K'=K(\Out^*(A_{\Gamma}), S')$.

\begin{prop}\label{Prop:NoTransvection}
Suppose $S'$ consists of only partial conjugations and $|S'|\geq 1$.
If there is a vertex $v\in\Gamma$ such that $\Gamma-\st(v)$ has at least three components, then $K'$ is connected.
Otherwise, $\Out^*(A_\Gamma)$ is isomorphic to the RAAG $A_{K'}$.
\end{prop}
\begin{proof}
Obviously, if there is a non-trivial partial conjugation by a vertex $v\in\Gamma$, then any two partial conjugations by $v$ commute.

Suppose there is a vertex $v$ such that $\Gamma-\st(v)$ has at least three components, and there is a non-trivial partial conjugation by $w$ for $w\neq v$.
By the first paragraph, it suffices to show that $P_v^C$ and $P_w^D$ commute for some $C$ and $D$.
If $(v,w)$ is not an SIL-pair, by Lemma~\ref{lem:commutingrelation}, any partial conjugation by $w$ commute with any partial conjugation by $v$. 
Otherwise, there is at least one shared component $C_1$ for the pair $(v,w)$.
If there is one more shared component $C_2$, by Lemma~\ref{lem:commutingrelation} we have $[P_v^{C_1},P_w^{C_2}]=1$. 
Otherwise, there is a subordinate component $C'$ of $\Gamma-\st(v)$, by Lemma~\ref{lem:commutingrelation}, we have $[P_v^{C'},P_w^{C_1}]=1$.

In \cite[Theorem B]{day2018subspace}, it is shown that $\mathrm{PSO}(A_\Gamma)$ is isomorphic to a RAAG if and only if the support graph of each vertex $v\in\Gamma$ is a forest, where the support graph is a simplicial graph whose vertices are components of $\Gamma-\st(v)$.
If there is no vertex $v$ such that $\Gamma-\st(v)$ has at least three components, then each support graph is a forest (either a single vertex or two vertices with or without an edge).
Therefore, $\Out^*(A_{\Gamma})$ is isomorphic to a RAAG, and by the results in \cite{day2018subspace} it can easily be seen that the defining graph of the RAAG is equal to $K'$.  
\end{proof}

Wiedmer showed in \cite{MW} that for any RAAG $A_\Lambda$, there exists a graph $\Gamma$ such that $A_\Lambda$, $\mathrm{PSO}(A_\Gamma)$ and $\Out^*(A_\Gamma)$ are all isomorphic.
In order to completely characterize non-relative hyperbolicity of $\Out^*(A_\Gamma)$ when $\Out^*(A_\Gamma)$ is isomorphic to $\mathrm{PSO}(A_\Gamma)$, thus, we need the following fundamental fact.

\begin{lem}\label{lem:RAAGcondition}
A RAAG $A_\Lambda$ is relatively hyperbolic if and only if its defining graph $\Lambda$ consists of either a single vertex or at least two components.
\end{lem}
\begin{proof}
If $\Lambda$ is a single vertex, then $A_\Lambda$ is isomorphic to $\Z$ and thus (relatively) hyperbolic.
If $\Lambda$ consists of at least two components, then $A_\Lambda$ is isomorphic to $A_{\Lambda_1}\ast \cdots \ast A_{\Lambda_n}$ where $\Lambda_i$'s are components of $\Lambda$; in particular, $A_\Lambda$ is relatively hyperbolic with respect to $\{A_{\Lambda_1},\cdots,A_{\Lambda_n}\}$. 

If $\Lambda$ is connected and has at least two vertices, then the commutativity graph is exactly the same as the defining graph by letting the generating set as the usual generators of RAAG.
\end{proof}

Now, we are ready to examine non-relative hyperbolicity of $\Out(A_\Gamma)$.

\begin{thm}\label{Main2}
If $\Out(A_{\Gamma})$ is infinite and not virtually a RAAG whose defining graph is either a single vertex or disconnected, then $\Out(A_{\Gamma})$ is not relatively hyperbolic.
\end{thm}
\begin{proof}
If $|S'|\leq 1$, then $\Out(A_\Gamma)$ is finite or has a finite-index subgroup isomorphic to $\mathbb{Z}$, and thus, it is (relatively) hyperbolic. 
If $\Gamma$ has only one vertex, then $\Out(A_\Gamma)$ is obviously finite. 
If $\Gamma$ has only two vertices, then $\Out(A_\Gamma)$ is isomorphic to $\GL_2(\Z)$, and thus, it is virtually the free group of rank 2.
Now we examine the commutativity graph $K' = K(\Out^*(A_{\Gamma}), S')$ for the case that $|S'|\geq 2$ and $\Gamma$ has at least 3 vertices.

If $S'$ does not have any transvection, by Proposition~\ref{Prop:NoTransvection} and Lemma~\ref{lem:RAAGcondition}, $\Out^*(A_\Gamma)$ is not relatively hyperbolic if and only if $\Out^*(A_\Gamma)$ is isomorphic to a RAAG whose defining graph is connected.

Now, we assume that there is at least one transvection in $S'$. 
\vspace{1mm}

\noindent\textbf{Claim A}. As long as they exist, any non-trivial partial conjugation and any transvection are joined by a path in $K'$ unless $\Out^*(A_{\Gamma})$ is isomorphic to $\Aut^*(\F_2)$.  
\vspace{1mm}

Let $R_{ab}$ be a transvection and suppose that there is a non-trivial partial conjugation $P_c^C$. 
As in the paragraph below Claim 1 in the proof of Theorem~\ref{Main1}, we have $[R_{ab},L_{ab}]=1$ whenever $R_{ab}$ is equal to $L_{ab}$ or not. Thus, we will show the existence of a path joining $R_{ab}$ and $P_c^C$ in $K'$. There are four cases, depending on $c$ and the adjacency of $a$ and $b$.
\vspace{0.7mm}

\textbf{\rom{1})} If $c=b$, then $[R_{ab},P_{c}^C]=1$ whenever $C$ contains $a$ or not.
\vspace{0.7mm}

\textbf{\rom{2})} Suppose that $c$ is neither $a$ nor $b$. 
If $a$ or $b$ is contained in $\lk(c)$, then $a\leq b$ implies that $b\in\lk(c)$ and thus $[P_c^C,R_{ab}]=1$ for any component $C$.
If $a$ and $b$ are in the same component $C'$ of $\Gamma- \st(c)$, then we have $[P_c^{C'},R_{ab}]=1$, which implies that the claim is true since $[P_c^{C'},P_c^C]=1$.
If $a$ and $b$ are contained in different components of $\Gamma - \st(c)$, then $a\leq b$ implies that $a\leq c$. Since $P_c^C$ and $R_{ac}$ are joined in $K'$ by an edge by the case \textbf{\rom{1}} and we have $[R_{ab},L_{ac}]=[L_{ac},R_{ac}]=1$, the claim holds.
\vspace{0.7mm}

\textbf{\rom{3})} Suppose $c=a$ and $a$ and $b$ are adjacent. If there is a non-trivial element $P_b^D$ in $\Out(A_{\Gamma})$, then $[P_a^C,P_b^D]=[P_b^D,R_{ab}]=1$, and thus, $P_c^C$ and $R_{ab}$ are joined by a path. 
Otherwise, there exists a component $C'$ of $\Gamma-\st(a)$ contained in $\lk(b)$, which implies $[R_{ab},P_a^{C'}]=1$ and thus the claim holds. 
\vspace{0.7mm}

\textbf{\rom{4})} Suppose $c=a$ but $a$ and $b$ are nonadjacent.
Since $a\leq b$, there is no subordinate component of $\Gamma-\st(a)$ for the pair $(a,b)$.
If $\Gamma-\st(a)$ has at least three components, then there are at least two shared components, say $C_1$ and $C_2$.
Since we have $[P_a^{C_1},P_b^{C_2}]=[P_b^{C_2},R_{ab}]=1$ by Lemma~\ref{lem:commutingrelation} and the case \textbf{\rom{1}}, the claim holds.

Now, suppose $\Gamma-\st(a)$ has two components, the shared component $C'$ and the dominating component $C''$.
If $\Gamma-\st(b)$ has a subordinate component $D$, then the claim holds since $[P_a^{C'},P_b^D]=[P_b^D,R_{ab}]=1$ by Lemma~\ref{lem:commutingrelation} and the case \textbf{\rom{1}}.
Otherwise, there are two situations depending on the existence of a vertex $x\in \lk(b)-\lk(a)$. 
If such a vertex $x$ exists, then $x\leq b$ and thus the claim holds since $[R_{ab},R_{xb}]=[R_{xb},P_a^{C'}]=1$. Otherwise, $a\sim b$ (in particular, $C''$ becomes $\{b\}$) and we have final two cases.
\begin{enumerate}
\item Suppose there is a vertex $d$ in $C'$, which defines a non-trivial partial conjugation in $\Out(A_\Gamma)$. 
If $\lk(a)\subseteq\lk(d)$, then $a\sim b\leq d$ and thus the claim holds since $[R_{ab},L_{ad}]=[L_{ad},R_{bd}]=1$ and there is a path joining $R_{bd}$ and $P_{a}^{C'}$ by the case \textbf{\rom{2}}.  
Otherwise, $C''$ is a subordinate component of $\Gamma-\st(a)$ for the pair $(a,d)$. By Lemma~\ref{lem:commutingrelation}, then any partial conjugation $P_d^{D}$ by $d$ commute with $P_a^{C''}$. Since there is a path joining $P_d^{D}$ and $R_{ab}$ by the case \textbf{\rom{2}}, the claim holds. 

\item Suppose there is no such a vertex $d$. In this case, only $a$, $b$, and vertices in $\lk(a)(=\lk(b))$ may be able to define non-trivial partial conjugations in $\Out(A_\Gamma)$. If $x\in \lk(a)$ does, then any non-trivial partial conjugation $P_x^X$ commutes with $P_c^C$ and $R_{ab}$, and thus the claim holds.
Otherwise, we need to see whether there is a transvection $R_{vw}$ different from $R_{ab}$ or $R_{ba}$. 

Suppose there is such a transvection $R_{vw}$. 
If $\{v,w\}\cap \{a,b\}=\emptyset$, then $[R_{ab},R_{vw}]=1$. Since there is a path joining $R_{vw}$ and $P_a^C$ by the case \textbf{\rom{2}}, the claim holds. 
If $v$ is either $a$ or $b$ (in particular, $a\sim b\leq w$ and $w\neq a,b$), since $w$ must not define a non-trivial partial conjugation, $w$ is adjacent to both $a$ and $b$. Since $[R_{ab},L_{aw}]=[L_{aw},R_{aw}]=1$ and there is a path joining $R_{aw}$ and $P_a^{C}$ by the case \textbf{\rom{3}}, the claim holds.
If $w$ is either $a$ or $b$ (in particular, $v\leq a\sim b$ and $v\neq a,b$), by the case \textbf{\rom{1}}, there is a path joining $R_{va}$ and $P_a^C$. Since $[R_{va},L_{vb}]=[L_{vb},R_{ab}]=1$, the claim holds.

Finally, if there is no such a transvection $R_{vw}$, then $a$ and $b$ are the only vertices defining non-trivial partial conjugations and each of $\Gamma-\st(a)$ and $\Gamma-\st(b)$ has two components. In particular, $S'$ consists of two partial conjugations and four (two right and two left) transvections.
In this case, $K'$ is discrete and $\Out^*(A_\Gamma)$ is isomorphic to $\Aut^*(\F_2)$. See Figure~\ref{Example2}.
\end{enumerate}

In summary, we checked that if there are a transvection and a partial conjugation in $S'$ which cannot be joined by a path in $K'$, then $\Out^*(A_\Gamma)$ is isomorphic to $\Aut^*(\mathbb{F}_2)$. 
\vspace{1mm}

\noindent\textbf{Claim B}. Every pair of transvections can be joined by a path in $K'$ unless $\Out^*(A_{\Gamma})$ is isomorphic to $\SL_2(\Z)$ or $\Aut^*(\F_2)$.
\vspace{1mm}

Since we assumed that $\Gamma$ has at least 3 vertices, the existence of a path in $K(\Aut^*(A_\Gamma),S)$ between two transvections appeared while proving Claim 1 in the proof of Theorem~\ref{Main1} tells us that of a path in $K(\Out^*(A_\Gamma),S')$ between the two transvections, except between $R_{ab}$ and $R_{ba}$; the path joining them in $K(\Aut^*(A_\Gamma),S)$ may use partial conjugations which have trivial images in $\Out(A_\Gamma)$. 

If there is another transvection $R_{vw}$, by the cases 1,2,3 and 4 of the proof of Claim 1 in the proof of Theorem~\ref{Main1}, there is a path joining $R_{vw}$ to $R_{ab}$ (and $R_{ba}$) in $K'$, and thus $R_{ab}$ and $R_{ba}$ are joined by a path.
If there is no other transvection but there is a non-trivial partial conjugation, by Claim A, there is a path joining $R_{ab}$ to $R_{ba}$ in $K'$ except when $\Out^*(A_\Gamma)$ is isomorphic to $\Aut^*(\F_2)$.
Lastly, if $S'=\{R_{ab},L_{ab},R_{ba},L_{ba}\}$, then any partial conjugation by $a$ or $b$ must be the identity in $\Out(A_\Gamma)$, which implies that $R_{ab}=L_{ab}$ and $R_{ba}=L_{ab}$. Thus $\Out^*(A_\Gamma)$ is isomorphic to $\Out^*(\F_2)$, which is $\SL_2(\Z)$.

\vspace{1mm}

By the previous two claims, as long as $|S'|>1$ and there is at least one transvection in $S'$, $K'$ is connected and by Theorem~\ref{Criterion} we get non-relative hyperbolicity of $\Out^*(A_\Gamma)$ unless $\Out^*(A_{\Gamma})$ is isomorphic to neither $\Aut^*(\F_2)$ nor $\SL_2(\Z)$. 
Since $\Aut^*(\F_2)$ is non-relative hyperbolic as explained in the paragraph below the proof of Theorem~\ref{Main1}, we conclude the proof. 
\end{proof}

We finish our paper by giving some examples of RAAGs such that their outer automorphism groups are relatively hyperbolic. If the graph $\Gamma$ is a cycle with $n$ vertices and $n$ edges, then $\Out(A_{\Gamma})$ is a finite group, as long as $n\geq 5$. 
If $\Gamma$ is the graph on the left in Figure \ref{Example}, then there are no transvections, and the central vertex induces the unique partial conjugation in $\Out(A_{\Gamma})$, which is thus virtually cyclic. 
If $\Gamma$ is the graph on the right in Figure \ref{Example}, then there is no partial conjugation in $\Out(A_{\Gamma})$, and the two equivalent vertices on the top and the bottom induce two transvections in $\Out(A_{\Gamma})$. By the argument in the last paragraph of the proof of Theorem \ref{Main2}, $\Out(A_{\Gamma})$ is virtually isomorphic to $\SL_2(\Z)$.

\begin{figure}[H]
\centering
\begin{tikzpicture}
\draw[black, thick] (-1/2,-0.866) -- (1/2,0.866);
\draw[black, thick] (-1/2,0.866) -- (1/2,-0.866);
\draw[black, thick] (-3/2,0.866) -- (3/2,0.866);
\draw[black, thick] (-3/2,-0.866) -- (3/2,-0.866);
\draw[black, thick] (-2,0) -- (-3/2,0.866);
\draw[black, thick] (-2,0) -- (-3/2,-0.866);
\draw[black, thick] (2,0) -- (3/2,0.866);
\draw[black, thick] (2,0) -- (3/2,-0.866);
\filldraw[black] (0,0) circle (1.5pt);
\filldraw[black] (2,0) circle (1.5pt);
\filldraw[black] (-2,0) circle (1.5pt);
\filldraw[black] (-3/2,0.866) circle (1.5pt);
\filldraw[black] (-1/2,0.866) circle (1.5pt);
\filldraw[black] (1/2,0.866) circle (1.5pt);
\filldraw[black] (3/2,0.866) circle (1.5pt);
\filldraw[black] (-3/2,-0.866) circle (1.5pt);
\filldraw[black] (-1/2,-0.866) circle (1.5pt);
\filldraw[black] (1/2,-0.866) circle (1.5pt);
\filldraw[black] (3/2,-0.866) circle (1.5pt);
\end{tikzpicture}
\qquad \qquad
\begin{tikzpicture}[scale = 1.3]
\draw[black, thick] (-1/3,0.65) -- (0,2);
\draw[black, thick] (1/3,0.65) -- (0,2);
\draw[black, thick] (1/3,0.65) -- (-1/3,0.65);
\draw[black, thick] (1/3,0.65) -- (2/3,1);
\draw[black, thick] (2/3,1) -- (0,2);
\draw[black, thick] (-1/3,0.65) -- (-2/3,1);
\draw[black, thick] (-2/3,1) -- (0,2);
\draw[black, thick] (1/3,0.65) -- (0,2);
\draw[black, thick] (0,5/4) -- (0,2);
\draw[black, thick] (0,5/4) -- (2/3,1);
\draw[black, thick] (0,5/4) -- (-2/3,1);
\draw[black, thick] (0,5/4) -- (0,0);
\draw[black, thick] (2/3,1) -- (0,0);
\draw[black, thick] (-2/3,1) -- (0,0);
\draw[black, thick] (1/3,0.65) -- (0,0);
\draw[black, thick] (-1/3,0.65) -- (0,0);
\filldraw[black] (0,2) circle (1pt);
\filldraw[black] (-1/3,0.65) circle (1pt);
\filldraw[black] (1/3,0.65) circle (1pt);
\filldraw[black] (-2/3,1) circle (1pt);
\filldraw[black] (2/3,1) circle (1pt);
\filldraw[black] (0,5/4) circle (1pt);
\filldraw[black] (0,0) circle (1pt);
\end{tikzpicture}
\hfill

\caption{Typical examples of $\Gamma$ with $\Out(A_{\Gamma})$ relatively hyperbolic.}
\label{Example}
\end{figure}

\begin{figure}[H]
\centering
\begin{tikzpicture}[scale = 1.3]
\draw[black, thick] (1,0) -- (0,0.3);
\draw[black, thick] (1,0) -- (0,-0.3);
\draw[black, thick] (-1,0) -- (0,0.3);
\draw[black, thick] (-1,0) -- (0,-0.3);
\draw[black, thick] (-2/3,1) -- (-1,0);
\draw[black, thick] (2/3,1) -- (1,0);
\draw[black, thick] (-2/3,1) -- (2/3,1);
\filldraw[black] (1,0) circle (1.5pt);
\filldraw[black] (-1,0) circle (1.5pt);
\filldraw[black] (0,0.3) circle (1.5pt);
\filldraw[black] (0,-0.3) circle (1.5pt);
\filldraw[black] (2/3,1) circle (1.5pt);
\filldraw[black] (-2/3,1) circle (1.5pt);
\end{tikzpicture}
\hfill
\caption{The graph $\Gamma$ where $\Out^*(A_\Gamma)$ is isomorphic to $\Aut^*(\F_2)$.}
\label{Example2}
\end{figure}

\medskip
\bibliographystyle{alpha} 
\bibliography{NRH}

\begin{thebibliography}{BDM09}

\bibitem[AAS07]{simplecriterion}
James~W Anderson, Javier Aramayona, and Kenneth~J Shackleton.
\newblock An obstruction to the strong relative hyperbolicity of a group.
\newblock {\em Journal of Group Theory}, 10(6):749--756, 2007.

\bibitem[BB97]{Bestvina1997MorseTA}
Mladen Bestvina and Noel Brady.
\newblock Morse theory and finiteness properties of groups.
\newblock {\em Inventiones mathematicae}, 129:445--470, 1997.

\bibitem[BDM09]{BDM09}
Jason Behrstock, Cornelia Druţu, and Lee Mosher.
\newblock Thick metric spaces, relative hyperbolicity, and quasi-isometric
  rigidity.
\newblock {\em Math. Ann.}, 344(543), 2009.

\bibitem[Bow12]{Bow}
B.~H. Bowditch.
\newblock Relatively hyperbolic groups.
\newblock {\em International Journal of Algebra and Computation},
  22(03):1250016, 2012.

\bibitem[Cha07]{Cha}
Ruth Charney.
\newblock An introduction to right-angled artin groups.
\newblock {\em Geom. Dedicata}, 125:141--158, 2007.

\bibitem[CV09]{charney2009finiteness}
Ruth Charney and Karen Vogtmann.
\newblock Finiteness properties of automorphism groups of right-angled artin
  groups.
\newblock {\em Bulletin of the London Mathematical Society}, 41(1):94--102,
  2009.

\bibitem[CV11]{charney2011subgroups}
Ruth Charney and Karen Vogtmann.
\newblock Subgroups and quotients of automorphism groups of raags.
\newblock {\em Low-dimensional and symplectic topology}, 82:9--27, 2011.

\bibitem[Dru09]{Drutu06relativelyhyperbolic}
Cornelia Druţu.
\newblock Relatively hyperbolic groups: geometry and quasi-isometric
  invariance.
\newblock {\em Commentarii Mathematici Helvetici}, 84(3):503--546, 2009.

\bibitem[DW18]{day2018subspace}
Matthew~B Day and Richard~D Wade.
\newblock Subspace arrangements, bns invariants, and pure symmetric outer
  automorphisms of right-angled artin groups.
\newblock {\em Groups, Geometry, and Dynamics}, 12(1):173--206, 2018.

\bibitem[Gen20]{genevois2020automorphisms}
Anthony Genevois.
\newblock Automorphisms of graph products of groups and acylindrical
  hyperbolicity, 2020.

\bibitem[KP14]{KP}
Nic Koban and Adam Piggott.
\newblock The bieri–neumann–strebel invariant of the pure symmetric
  automorphisms of a right-angled artin group.
\newblock {\em Illinois J. Math.}, 58:27--41, 2014.

\bibitem[Lau95]{Laurence}
Michael~R. Laurence.
\newblock A generating set for the automorphism group of a graph group.
\newblock {\em Journal of the London Mathematical Society}, 52(2):318--334,
  1995.

\bibitem[Osi06]{Osin06}
Denis Osin.
\newblock Relatively hyperbolic groups: intrinsic geometry, algebraic
  properties, and algorithmic problems.
\newblock {\em Memoirs of the American Mathematical Society}, 179, 2006.

\bibitem[Ser89]{SERVATIUS}
Herman Servatius.
\newblock Automorphisms of graph groups.
\newblock {\em Journal of Algebra}, 126(1):34--60, 1989.

\bibitem[Wie21]{MW}
Manuel Wiedmer.
\newblock Right-angled artin groups as finite-index subgroups of their outer
  automorphism groups.
\newblock Master thesis, ETH Zurich, Zurich, 2022-02-21.

\end{thebibliography}

\end{document}